\documentclass[12pt]{article}
\usepackage{fullpage}
\usepackage{amsmath,amssymb,amsthm}
\usepackage{boxedminipage}
\usepackage{graphicx}

\newcommand{\reals}{\mathbb{R}}
\newcommand{\minimize}{\mathop{\mathrm{minimize}{}}}

\newcommand{\argmin}{\mathop{\mathrm{arg\,min}{}}}

\newcommand{\dom}{\mathrm{dom}}
\newcommand{\prox}{\mathrm{prox}}
\newcommand{\Prob}{\mathrm{Prob}}
\newcommand{\E}{\mathbb{E}}
\newcommand{\eqdef}{\stackrel{\mathrm{def}}{=}}
\newcommand{\Lmax}{L_\mathrm{max}}
\newcommand{\Lavg}{L_\mathrm{avg}}

\newtheorem{assumption}{Assumption}

\newtheorem{lemma}{Lemma}

\newtheorem{theorem}{Theorem}
\newtheorem{corollary}{Corollary}

\title{A Proximal Stochastic Gradient Method with Progressive Variance Reduction}
\author{Lin Xiao\thanks{
    Machine Learning Group, Microsoft Research, Redmond, WA 98052.
    Email: \texttt{lin.xiao@microsoft.com}. }
\and Tong Zhang\thanks{    
    Department of Statistics, Rutgers University, Piscataway, NJ 08854;
    and Baidu Inc., Beijing 100085.
    Email: \texttt{tzhang@stat.rutgers.edu}. } }

\date{March 18, 2014}

\begin{document}
\maketitle

\vspace{-4ex}
\begin{abstract}
We consider the problem of minimizing the sum of two convex functions: 
one is the average of a large number of smooth component functions, and 
the other is a general convex function that admits a simple proximal mapping. 
We assume the whole objective function is strongly convex.
Such problems often arise in machine learning, known as regularized empirical 
risk minimization. We propose and analyze a new proximal stochastic gradient 
method, which uses a multi-stage scheme to progressively reduce the variance 
of the stochastic gradient. 
While each iteration of this algorithm has similar cost as the classical 
stochastic gradient method (or incremental gradient method), 
we show that the expected objective value converges to the optimum
at a geometric rate. 
The overall complexity of this method is much lower than both the 
proximal full gradient method and the standard 
proximal stochastic gradient method.
%
%
\end{abstract}

\section{Introduction}
We consider the problem of minimizing the sum of two convex functions:
\begin{equation}\label{eqn:min-composite}
    \minimize_{x\in\reals^d} \quad \{P(x) \eqdef F(x) + R(x)\},
\end{equation}
where $F(x)$ is the average of many smooth component functions~$f_i(x)$, i.e., 
\begin{equation}\label{eqn:avg-loss}
    F(x) = \frac{1}{n} \sum_{i=1}^n f_i(x) ,
\end{equation}
and $R(x)$ is relative simple but can be non-differentiable.
We are especially interested in the case where the number of components~$n$
is very large, and it can be advantageous to use incremental methods
(such as stochastic gradient method)
that operate on a single component~$f_i$ at each iteration, rather than
on the entire cost function.

Problems of this form often arise in machine learning and statistics, 
known as \emph{regularized empirical risk minimization};
see, e.g., \cite{HTFbook}.
In such problems, we are given a collection of training examples 
$(a_1, b_1), \ldots, (a_n, b_n)$, where each $a_i\in\reals^d$ is a 
feature vector and $b_i\in\reals$ is the desired response. 
For least-squares regression, the component loss functions are
$f_i(x)=(1/2)(a_i^T x-b_i)^2$, 
and popular choices of the regularization term include 
$R(x)=\lambda_1\|x\|_1$ (the Lasso), 
$R(x)=(\lambda_2/2)\|x\|_2^2$ (ridge regression), or
$R(x)=\lambda_1\|x\|_1 + (\lambda_2/2)\|x\|_2^2$ (elastic net),
where $\lambda_1$ and $\lambda_2$ are nonnegative regularization parameters.
For binary classification problems, each $b_i\in\{+1,-1\}$ is the desired 
class label, and a popular loss function is the logistic loss
$f_i(x)=\log(1+\exp(-b_i a_i^T x))$,
which can be combined with any of the regularization terms mentioned above.

The function $R(x)$ can also be used to model convex constraints. 
Given a closed convex set~$C\subseteq \reals^d$, the constrained problem
\[
    \minimize_{x\in C} \quad \frac{1}{n}\sum_{i=1}^n f_i(x)
\]
can be formulated as~\eqref{eqn:min-composite} by setting $R(x)$ to be
the indicator function of~$C$, i.e., $R(x)=0$ if $x\in C$ and 
$R(x)=\infty$ otherwise.
Mixtures of the ``soft'' regularizations (such as $\ell_1$ or $\ell_2$
penalties) and ``hard'' constraints are also possible.

The results presented in this paper are based on the following 
assumptions.
\begin{assumption}\label{asp:smooth-components}
The function $R(x)$ is lower semi-continuous and convex,
and its effective domain, $\dom(R):=\{x\in\reals^d\,|\,R(x)<+\infty\}$,
is closed.
Each $f_i(x)$, for $i=1,\ldots,n$, is differentiable
on an open set that contains~$\dom(R)$, and their gradients are
Lipschitz continuous.
That is, there exist~$L_i>0$ such that for all $x, y\in\dom(R)$,
\begin{equation}\label{eqn:smooth-i}
    \|\nabla f_i(x) - \nabla f_i(y)\| \leq L_i \|x-y\|. 
\end{equation}
\end{assumption}
Assumption~\ref{asp:smooth-components} implies that the gradient 
of the average function $F(x)$ is also Lipschitz continuous, i.e., 
there is an $L>0$ such that for all $x, y\in\dom(R)$,
\[
    \|\nabla F(x) - \nabla F(y)\| \leq L \|x-y\|.
\]
Moreover, we have $L\leq (1/n)\sum_{i=1}^n L_i$.
\begin{assumption}\label{asp:strong-convex}
The overall cost function $P(x)$ is strongly convex, i.e., 
there exist $\mu>0$ such that for all $x\in\dom(R)$ and $y\in\reals^d$,
\begin{equation}\label{eqn:strong-convex}
    P(y) \geq P(x) + \xi^T(y-x) + \frac{\mu}{2} \|y-x\|^2,
    \quad \forall\, \xi\in\partial P(x).
\end{equation}
\end{assumption}
The \emph{convexity parameter} of a function is the largest~$\mu$
such that the above condition holds.
The strong convexity of $P(x)$ may come from either $F(x)$ or $R(x)$ or both.
More precisely, let $F(x)$ and  $R(x)$ have convexity parameters $\mu_F$ 
and $\mu_R$ respectively, then $\mu \geq \mu_F + \mu_R$.
We note that it is possible to have $\mu>L$ 
although we must have $\mu_F\leq L$.

\subsection{Proximal gradient and stochastic gradient methods}

A standard method for solving problem~\eqref{eqn:min-composite} is 
the \emph{proximal gradient method}.
Given an initial point $x_0\in\reals^d$, the proximal gradient method uses 
the following update rule for $k=1,2,\ldots$
\[
    x_k = \argmin_{x\in\reals^d} \left\{ \nabla F(x_{k-1})^T x
    + \frac{1}{2\eta_k}\|x-x_{k-1}\|^2 + R(x) \right\},
\]
where $\eta_k$ is the step size at the $k$-th iteration.
Throughout this paper, we use $\|\cdot\|$ to denote the usual Euclidean norm, 
i.e., $\|\cdot\|_2$, unless otherwise specified.
With the definition of \emph{proximal mapping}
\[
\prox_R(y) = \argmin_{x\in\reals^d}\left\{ \frac{1}{2}\|x-y\|^2 + R(x)\right\}, 
\]
the proximal gradient method can be written more compactly as
\begin{equation}\label{eqn:prox-fg}
x_k = \prox_{\eta_k R}\bigl(x_{k-1} - \eta_k \nabla F(x_{k-1}) \bigr) .
\end{equation}
This method can be viewed as a special case of \emph{splitting} algorithms 
\cite{LiM:79,ChR:97,Tse:00}, and its accelerated variants have been proposed
and analyzed in \cite{BeckTeboulle09,Nesterov13composite}.

When the number of components~$n$ is very large, each iteration
of~\eqref{eqn:prox-fg} can be very expensive
since it requires computing the gradients for all the~$n$ 
component functions~$f_i$, and also their average.
For this reason, we refer to~\eqref{eqn:prox-fg}
as the proximal \emph{full} gradient (Prox-FG) method.
An effective alternative is the \emph{proximal stochastic gradient} 
(Prox-SG) method:
at each iteration $k=1,2,\ldots$, we draw~$i_k$ randomly from $\{1,\ldots,n\}$
and take the update
\begin{equation}\label{eqn:prox-sg}
x_k = \prox_{\eta_k R}\bigl(x_{k-1} - \eta_k \nabla f_{i_k}(x_{k-1}) \bigr) .
\end{equation}
Clearly we have $\E \nabla f_{i_k}(x_{k-1})=\nabla F(x_{k-1})$.
The advantage of the Prox-SG method is that at each iteration, 
it only evaluates gradient of a single component function,
thus the computational cost per iteration is only $1/n$ that of 
the Prox-FG method.
However, due to the variance introduced by random sampling,
the Prox-SG method converges much more slowly than the Prox-FG method.
To have a fair comparison of their overall computational cost, 
we need to combine their cost per iteration and iteration complexity.

Let $x^\star=\argmin_x P(x)$.
Under the Assumptions~\ref{asp:smooth-components} and~\ref{asp:strong-convex},
the Prox-FG method with a constant step size $\eta_k=1/L$ generates
iterates that satisfy
\begin{equation}\label{eqn:prox-fg-rate}
P(x_k)-P(x_\star) \leq O\biggl(\left(\frac{L-\mu_F}{L+\mu_R}\right)^k \biggr).
\end{equation}
(See Appendix~\ref{apd:prox-fg} for a proof of this result.)
The most interesting case for large-scale applications is when $\mu\ll L$,
and the ratio $L/\mu$ is often called the \emph{condition number} of the 
problem~\eqref{eqn:min-composite}.
In this case, the Prox-FG method needs 
$O\left((L/\mu)\log(1/\epsilon)\right)$ iterations to ensure
$P(x_k)-P(x_\star)\leq \epsilon$.
Thus the overall complexity of Prox-FG, 
in terms of the total number of component gradients evaluated to find an
$\epsilon$-accurate solution, is $O\left(n(L/\mu)\log(1/\epsilon)\right)$.
The accelerated Prox-FG methods in \cite{BeckTeboulle09, Nesterov13composite}
reduce the complexity to $O\bigl(n\sqrt{L/\mu}\log(1/\epsilon)\bigr)$.

On the other hand, with a diminishing step size $\eta_k = 1/(\mu k)$,
the Prox-SG method converges at a sublinear rate 
(\cite{DuS:09,LLZ:09}):
\begin{equation}\label{eqn:prox-sg-rate}
    \E P(x_{k}) - P(x_\star) \leq O\left(1/\mu k\right).
\end{equation}
Consequently, the total number of component gradient evaluations required by
the Prox-SG method to find an $\epsilon$-accurate solution (in expectation)
is $O(1/\mu\epsilon)$. 
This complexity scales poorly in~$\epsilon$ compared with 
$\log(1/\epsilon)$, but it is independent of~$n$. 
Therefore, when~$n$ is very large, the Prox-SG method can be more efficient,
especially to obtain low-precision solutions.

There is also a vast literature on \emph{incremental gradient methods} for 
minimizing the sum of a large number of component functions.
The Prox-SG method can be viewed as a variant of the randomized
incremental proximal algorithms proposed in \cite{Bertsekas11}.
Asymptotic convergence of such methods typically requires diminishing
step sizes and only have sublinear convergence rates. 
A comprehensive survey on this topic can be found in \cite{Bertsekas10}.

\subsection{Recent progresses and our contributions}

Both the Prox-FG and Prox-SG methods do not fully exploit the problem structure
defined by~\eqref{eqn:min-composite} and~\eqref{eqn:avg-loss}.
In particular, Prox-FG ignores the fact that the smooth part $F(x)$
is the average of~$n$ component functions. On the other hand,
Prox-SG can be applied for more general stochastic optimization problems, 
and it does not exploit the fact that the objective function 
in~\eqref{eqn:min-composite} is actually a deterministic function.
Such inefficiencies in exploiting problem structure 
leave much room for further improvements.

Several recent work considered various special cases 
of~\eqref{eqn:min-composite} and~\eqref{eqn:avg-loss},
and developed algorithms that enjoy the complexity 
(total number of component gradient evaluations)
\begin{equation}\label{eqn:complexity-max}
O\bigl( (n+\Lmax/\mu)\log(1/\epsilon) \bigr),
\end{equation}
where $\Lmax=\max\{L_1,\ldots, L_n\}$.
If $\Lmax$ is not significantly larger than~$L$, this complexity 
is far superior than that of both the Prox-FG and Prox-SG methods.
In particular, Shalev-Shwartz and Zhang 
\cite{SSZhang13SDCA,SSZhang12prox-SDCA}
considered the case where the component functions have the form 
$f_i(x)=\phi_i(a_i^T x)$ and the Fenchel conjugate functions
of~$\phi_i$ and~$R$ can be computed efficiently.
With the additional assumption that $R(x)$ itself is $\mu$-strongly convex,
they showed that a proximal stochastic dual coordinate ascent (Prox-SDCA) 
method achieves the complexity in~\eqref{eqn:complexity-max}.

Le Roux et al.\ \cite{LeRouxSchmidtBach12} considered the case where 
$R(x)\equiv 0$,
and proposed a \emph{stochastic average gradient} (SAG) method which 
has complexity $O\bigl(\max\{n,\Lmax/\mu\}\log(1/\epsilon)\bigr)$.
Apparently this is on the same order as~\eqref{eqn:complexity-max}.
The SAG method is a randomized variant of the 
\emph{incremental aggregated gradient} method of 
Blatt et al.\ \cite{BlattHeroGauchman07}, and needs to store
the most recent gradient for each component function~$f_i$, which is $O(nd)$.
While this storage requirement can be prohibitive for large-scale problems,
it can be reduced to $O(n)$ for problems with more favorable structure, 
such as linear prediction problems in machine learning.

More recently, Johnson and Zhang \cite{JohnsonZhang13} developed 
another algorithm for the case $R(x)\equiv 0$, called
\emph{stochastic variance-reduced gradient} (SVRG).
The SVRG method employs a multi-stage scheme to progressively reduce the 
variance of the stochastic gradient, and achieves the same low complexity
in~\eqref{eqn:complexity-max}.
Moreover, it avoids storage of past gradients for the component functions,
and its convergence analysis is considerably simpler than that of SAG.
A very similar algorithm was proposed 
by Zhang et al.\ \cite{ZhangMahdaviJin13}, but with a
worse convergence rate analysis.
Another recent effort to extend the SVRG method is
\cite{KonecnyRichtarik13}.

In this paper, we extend the variance reduction technique of SVRG to
develop a proximal SVRG (Prox-SVRG) method for solving the more general 
class of problems defined in~\eqref{eqn:min-composite} and~\eqref{eqn:avg-loss}.
We show that with uniform sampling of the component functions, 
the Prox-SVRG method achieves the same complexity in~\eqref{eqn:complexity-max}.
Moreover, our method incorporates a weighted sampling strategy. 
When the sampling probabilities for~$f_i$ are proportional to their Lipschitz
constants~$L_i$, the Prox-SVRG method has complexity
\begin{equation}\label{eqn:complexity-avg}
O\bigl( (n+\Lavg/\mu)\log(1/\epsilon) \bigr),
\end{equation}
where $\Lavg=(1/n)\sum_{i=1}^n L_i$.
This bound improves upon the one in~\eqref{eqn:complexity-max},
especially for applications where the component functions vary 
substantially in smoothness.

\section{The Prox-SVRG method}

Recall that in the Prox-SG method~\eqref{eqn:prox-sg},
with uniform sampling of~$i_k$, we have unbiased estimate
of the full gradient at each iteration.
In order to ensure asymptotic convergence,
the step size $\eta_k$ has to decay to zero to mitigate the effect of variance
introduced by random  sampling, which leads to slow convergence.
However, if we can gradually reduce the variance in estimating the full
gradient, then it is possible to use much larger (even constant) 
step sizes and obtain much faster convergence rate. 
Several recent work 
(e.g., \cite{FriedlanderSchmidt12, ByrdChinNocedalWu12,FriedlanderGoh:2013})
have explored this idea by using mini-batches with exponentially growing sizes,
but their overall computational cost 
is still on the same order as full gradient methods.

Instead of increasing the batch size gradually, 
we use the variance reduction technique of SVRG \cite{JohnsonZhang13}, 
which computes the full batch periodically.
More specifically, we maintain an estimate~$\tilde x$ of the optimal 
point~$x_\star$, which is updated periodically, 
say after every~$m$ Prox-SG iterations.
Whenever~$\tilde x$ is updated, we also computes the full gradient
\[
\nabla F(\tilde x) = \frac{1}{n}\sum_{i=1}^n \nabla f_i(\tilde x),
\]
and use it to modify the next~$m$ stochastic gradient directions.
Suppose the next~$m$ iterations are initialized with $x_0=\tilde{x}$
and indexed by $k=1,\ldots,m$.
For each $k\geq 1$, we first randomly pick $i_k\in\{1,\ldots,n\}$ and compute
\[
v_k = \nabla f_{i_k}(x_{k-1})-\nabla f_{i_k}(\tilde{x}) + \nabla F(\tilde{x}),
\]
then we replace $\nabla f_{i_k}(x_{k-1})$ in the Prox-SG 
method~\eqref{eqn:prox-sg} with~$v_k$, i.e., 
\begin{equation}\label{eqn:prox-svrg-update}
x_k = \prox_{\eta_k R}\bigl(x_{k-1} - \eta_k v_k \bigr) .
\end{equation}

Conditioned on $x_{k-1}$, we can take expectation with respect to~$i_k$ and 
obtain
\begin{eqnarray*}
\E v_k &=& \E \nabla f_{i_k}(x_{k-1}) - \E \nabla f_{i_k}(\tilde{x}) 
           + \nabla F(\tilde{x}) \\
       &=& \nabla F(x_{k-1}) - \nabla F(\tilde{x}) + \nabla F(\tilde{x})\\
       &=& \nabla F(x_{k-1}).
\end{eqnarray*}
Hence, just like $\nabla f_{i_k}(x_{k-1})$, the modified direction
$v_k$ is also a stochastic gradient of~$F$ at~$x_{k-1}$.
However, the variance $\E\|v_k-\nabla F(x_{k-1})\|^2$ can be much smaller
than $\E\|\nabla f_{i_k}(x_{k-1})-\nabla F(x_{k-1})\|^2$.
In fact we will show in Section~\ref{sec:var-bd} 
that the following inequality holds:
\begin{equation}\label{eqn:var-bd-max}
\E \|v_k-\nabla F(x_{k-1})\|^2 \leq 4 \Lmax \bigl[ P(x_{k-1})-P(x_\star)
 + P(\tilde{x}) - P(x_\star) \bigr].
\end{equation}
Therefore, when both $x_{k-1}$ and $\tilde x$ converge to $x_\star$, 
the variance of $v_k$ also converges to zero.
As a result, we can use a constant step size and obtain much faster convergence.

\begin{figure}[t]
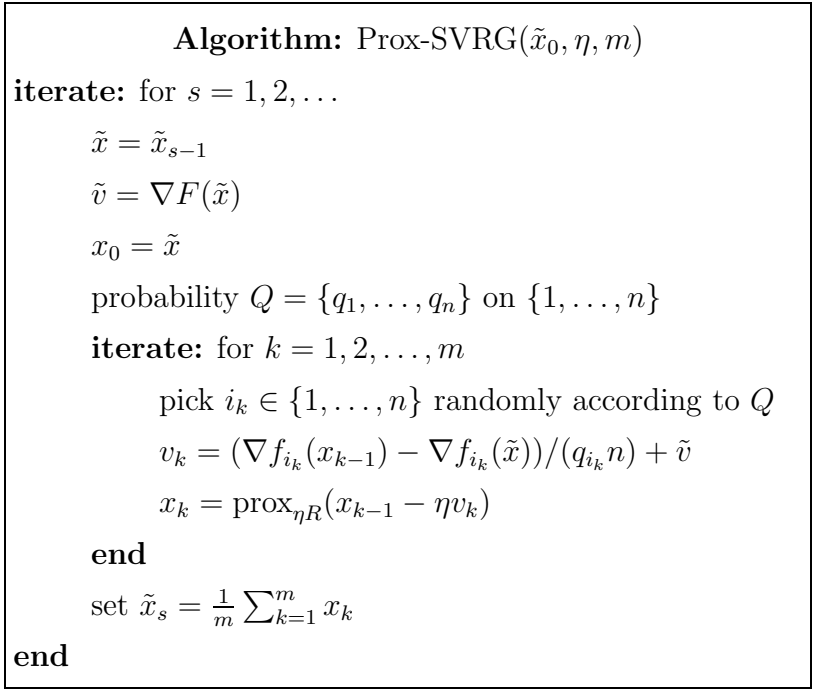

\begin{center}
\begin{boxedminipage}{0.65\textwidth}
\vspace{1ex}
\centerline{\textbf{Algorithm:} Prox-SVRG$(\tilde{x}_0, \eta, m)$}
\vspace{1ex}
\textbf{iterate:} for $s=1,2,\ldots$
\vspace{-1ex}
\begin{itemize} \itemsep 0pt
  \item[] $\tilde{x}=\tilde{x}_{s-1}$
  \item[] $\tilde{v} = \nabla F(\tilde{x})$
  \item[] $x_0=\tilde{x}$
  \item [] probability $Q=\{q_1,\ldots,q_n\}$ on $\{1,\ldots,n\}$
  \item[] \textbf{iterate:} for $k=1,2,\ldots,m$
    \vspace{-1ex}
    \begin{itemize}
      \item[] pick $i_k\in\{1,\ldots,n\}$ randomly according to $Q$
      \item[] $v_k = (\nabla f_{i_k}(x_{k-1}) - \nabla f_{i_k}(\tilde{x}))/(q_{i_k} n)  + \tilde{v}$
      \item[] $x_k = \prox_{\eta R} (x_{k-1} - \eta v_k)$
    \end{itemize}
\vspace{-1ex}
\item[] \textbf{end}
\item[] set $\tilde{x}_s = \frac{1}{m} \sum_{k=1}^m x_k$
\end{itemize}
\vspace{-2ex}
\item[] \textbf{end}
\vspace{1ex}
\end{boxedminipage}
\end{center}
\label{fig:prox-svrg}
\vspace{-2ex}
\caption{The Prox-SVRG method.}
\end{figure}

Figure~\ref{fig:prox-svrg} gives the full description of the 
Prox-SVRG method with a constant step size~$\eta$.
It allows random sampling from a general distribution $\{q_1,\ldots,q_n\}$,
thus is more flexible than the uniform sampling scheme described above.
It is not hard to verify that the modified stochastic gradient,
\begin{equation}\label{eqn:modified-sg}
v_k = (\nabla f_{i_k}(x_{k-1}) - \nabla f_{i_k}(\tilde{x}))/(q_{i_k}n) 
+ \nabla F(\tilde{x}) ,
\end{equation}
still satisfies $\E v_k=\nabla F(x_{k-1})$.
In addition, its variance can be bounded similarly as in~\eqref{eqn:var-bd-max}
(see Corollary~\ref{cor:var-bd} in Section~\ref{sec:var-bd}).

The Prox-SVRG method uses a multi-stage scheme to progressively reduce
the variance of the modified stochastic gradient~$v_k$ 
as both $\tilde{x}$ and $x_{k-1}$ converges to $x_\star$.
Each stage~$s$ requires $n+2m$ component gradient evaluations:
$n$ for the full gradient at the beginning of each stage, and two for 
each of the~$m$ proximal stochastic gradient steps.
For some problems such as linear prediction in machine learning, 
the cost per stage can be further reduced to only
$n+m$ gradient evaluations.
In practical implementations, we can also set $\tilde{x}_s$ to be the last 
iterate~$x_m$, instead of $(1/m)\sum_{k=1}^m x_k$, of the previous stage.
This simplifies the computation and we did not observe much difference in
the convergence speed.

\section{Convergence analysis}
\label{sec:analysis}

\begin{theorem}\label{thm:prox-svrg}
Suppose Assumptions~\ref{asp:smooth-components} and~\ref{asp:strong-convex} 
hold, and let $x_\star=\argmin_x P(x)$ and $L_Q= \max_i L_i/(q_i n)$.
In addition, assume that $0<\eta<1/(4L_Q)$ 
and~$m$ is sufficiently large so that
\begin{equation}\label{eqn:converge-factor}
    \rho=\frac{1}{\mu\eta(1-4L_Q\eta)m}+\frac{4L_Q\eta(m+1)}{(1-4L_Q\eta)m}<1.
\end{equation}
Then the Prox-SVRG method in Figure~\ref{fig:prox-svrg}
has geometric convergence in expectation:
\[
    \E P(\tilde{x}_s) - P(x_\star) \leq \rho^s [P(\tilde{x}_0) - P(x_\star)].
\]
\end{theorem}

We have the following remarks regarding the above result:
\begin{itemize}
\item 
The ratio $L_Q/\mu$ can be viewed as a ``weighted'' condition number of $P(x)$.
Theorem~\ref{thm:prox-svrg} implies that setting~$m$ to be on the same
order as $L_Q/\mu$ is sufficient to have geometric convergence.
To see this, let $\eta=\theta/L_Q$ with $0<\theta<1/4$.
When $m\gg 1$, we have
\[
  \rho\approx \frac{L_Q/\mu}{\theta(1-4\theta)m} + \frac{4\theta}{1-4\theta}.
\]
As a result, choosing $\theta=0.1$ and $m=100(L_Q/\mu)$ results in 
$\rho\approx 5/6$.

\item
In order to satisfy $\E P(\tilde{x}_s) - P(x_\star) \leq \epsilon$,
the number of stages~$s$ needs to satisfy
\[
    s\geq \log\rho^{-1} \log \frac{P(\tilde{x}_0)-P(x_\star)}{\epsilon}.
\]
Since each stage requires $n+2m$ component gradient evaluations, 
and it is sufficient to set~$m=\Theta(L_Q/\mu)$,
the overall complexity is
\[
O\bigl( (n+L_Q/\mu)\log(1/\epsilon) \bigr).
\]

\item 
For uniform sampling, $q_i=1/n$ for all $i=1,\ldots,n$, so we have
$L_Q=\max_i L_i$ and the above complexity bound
becomes~\eqref{eqn:complexity-max}.

The smallest possible value for $L_Q$ is $L_Q= (1/n)\sum_{i=1}^n L_i$, 
achieved at $q_i = L_i/\sum_{j=1}^n L_j$, i.e., when the sampling probabilities
for the component functions are proportional to their Lipschitz constants.
In this case, the above complexity bound becomes~\eqref{eqn:complexity-avg}.
\end{itemize}

Since $P(\tilde{x}_s) - P(x_\star) \geq 0$, Markov's
inequality and Theorem~\ref{thm:prox-svrg} imply 
that for any $\epsilon>0$,
\[
\Prob\Bigl( P(\tilde{x}_s) - P(x_\star) \geq \epsilon\Bigr)
~\leq~ \frac{\E [P(\tilde{x}_s) - P(x_\star)]}{\epsilon} 
~\leq~ \frac{\rho^s[P(\tilde{x}_0) - P(x_\star)] }{\epsilon}.
\]
Thus we have the following high-probability bound. 
\begin{corollary}\label{cor:high-prob}
Suppose the assumptions in Theorem~\ref{thm:prox-svrg} hold. 
Then for any $\epsilon>0$ and $\delta\in(0,1)$, we have
\[
  \Prob\bigl(P(\tilde{x}_s) - P(x_\star) \leq \epsilon \bigr) \geq 1-\delta
\]
provided that the number of stages~$s$ satisfies
\[
s \geq \log\left(\frac{[P(\tilde{x}_0) - P(x_\star)]}{\delta\epsilon}\right)
\bigg/\log\left(\frac{1}{\rho}\right).
\]
\end{corollary}

If $P(x)$ is convex but not strongly convex, then for any $\epsilon>0$, we can define
\[
P_\epsilon(x) = F(x) + R_\epsilon(x) , \qquad 
R_\epsilon(x) = \frac{\epsilon}{2} \|x\|^2 + R(x) .
\]
It follows that $P_\epsilon(x)$ is $\epsilon$-strongly convex. 
We can apply the Prox-SVRG method in Figure~\ref{fig:prox-svrg}
to $P_\epsilon(x)$, which replaces the update formula for $x_k$ by the following update rule:
\[
x_k =\prox_{\eta R_\epsilon} (x_{k-1} - \eta v_k) 
= \argmin_{x \in\reals^d} \left\{ \frac{1}{2} 
  \left\|x-\frac{1}{1+\eta \epsilon} (x_{k-1}-\eta v_k) \right\|^2 
  + \frac{\eta}{1+\eta\epsilon} R(x) \right\} .
\]
Theorem~\ref{thm:prox-svrg} implies the following result.
\begin{corollary}\label{cor:non-sc}
  Suppose Assumption~\ref{asp:smooth-components} holds and let $L_Q=\max_i L_i/(q_i n)$.
  In addition, assume that $0< \eta < 1/(4L_Q)$ and $m$ is sufficiently large so that
  \[
  \rho=\frac{1}{\epsilon \eta(1-4L_Q\eta)m}+\frac{4L_Q\eta(m+1)}{(1-4L_Q\eta)m}<1.
  \]
  Then the Prox-SVRG method in Figure~\ref{fig:prox-svrg}, applied to $P_\epsilon(x)$, achieves
\[
\E P(\tilde{x}_s) \leq  \min_x [ P(x) + (\epsilon/2) \|x\|^2] + \rho^s [P(\tilde{x}_0) + (\epsilon/2) \|\tilde{x}_0\|^2] .
\]
\end{corollary}

If $P(x)$ has a minimum and it is achieved by some $x_\star\in\dom(R)$, then 
Corollary~\ref{cor:non-sc} implies
\[
 \E P(\tilde{x}_s) - P(x_\star) \leq  (\epsilon/2) \|x_\star\|^2
+ \rho^s [P(\tilde{x}_0) + (\epsilon/2) \|\tilde{x}_0\|^2] .
\]
This result means that if we take $m= O(L_Q/\epsilon)$ 
and $s\geq\log (1/\epsilon)/\log(1/\rho)$, then 
\[
    \E P(\tilde{x}_s)-P(x_\star) \leq  \epsilon\,[ P(\tilde{x}_0) 
    +(1/2)\|x_\star\|^2 + (\epsilon/2)\|\tilde{x}_0\|^2 ]
\]
The overall complexity 
(in terms of the number of component gradient evaluations) is 
\[
    O\bigl( (n+L_Q/\epsilon)\log(1/\epsilon)\bigr).
\]
Similar results for the case of $R(x)\equiv 0$
have been obtained in 
\cite{LeRouxSchmidtBach12,MahdaviZhangJin13,KonecnyRichtarik13}.
We can also derive a high-probability bound based on 
Corollary~\ref{cor:high-prob}, but omit the details here.

\subsection{Bounding the variance}
\label{sec:var-bd}
Our bound on the variance of the modified stochastic gradient~$v_k$ 
is a corollary of the following lemma.
\begin{lemma}\label{lem:var-bd}
Consider $P(x)$ as defined in~\eqref{eqn:min-composite} 
and~\eqref{eqn:avg-loss}.
Suppose Assumption~\ref{asp:smooth-components} holds,
and let $x_\star=\argmin_x P(x)$ and $L_Q= \max_i L_i/(q_i n)$.
Then
\[
\frac{1}{n}\sum_{i=1}^n \frac{1}{n q_i} \|\nabla f_i(x)-\nabla f_i(x_\star)\|^2
   \leq 2 L_Q\left[ P(x) - P(x_\star)\right] .
\]
\end{lemma}

\begin{proof}
Given any $i\in\{1,\ldots,n\}$, consider the function
\[
    \phi_i(x) = f_i(x) - f_i(x_\star) - \nabla f_i(x_\star)^T(x-x_\star).
\]
It is straightforward to check that $\nabla \phi_i(x_\star)=0$,
hence $\min_x \phi_i(x) =\phi_i(x_\star)= 0$. 
Since $\nabla \phi_i(x)$ is Lipschitz continuous with constant~$L_i$, we have
(see, e.g., \cite[Theorem~2.1.5]{Nesterov04book})
\[
    \frac{1}{2L_i} \|\nabla \phi_i(x)\|^2 \leq \phi_i(x)  - \min_{y} \phi_i(y)
    = \phi_i(x) - \phi_i(x_\star) = \phi_i(x). 
\]
This implies
\[
    \|\nabla f_i(x) - \nabla f_i(x_\star)\|^2 \leq 
    2L_i \left[f_i(x) - f_i(x_\star) - \nabla f_i(x_\star)^T(x-x_\star)\right].
\]
By dividing the above inequality by $1/(n^2 q_i)$, 
and summing over $i=1,\ldots,n$, we obtain
\[
\frac{1}{n}\sum_{i=1}^n\frac{1}{n q_i}\|\nabla f_i(x) - \nabla f_i(x_\star)\|^2
\leq 2L_Q \left[ F(x) - F(x_\star) - \nabla F(x_\star) (x-x_\star)\right].
\]
By the optimality of $x_\star$, i.e., 
\[
    x_\star=\argmin_x P(x) = \argmin_x \left\{ F(x)+R(x)\right\},
\]
there exist $\xi_\star\in\partial R(x_\star)$ such that 
$\nabla F(x_\star) + \xi_\star = 0$.
Therefore
\begin{eqnarray*}
 F(x) - F(x_\star) - \nabla F(x_\star) (x-x_\star) 
&=& F(x) - F(x_\star) + \xi_\star (x-x_\star) \\
&\leq& F(x) - F(x_\star) + R(x) - R(x_\star) \\
&=&  P(x) - P(x_\star),
\end{eqnarray*}
where in the last inequality, we used convexity of $R(x)$.
This proves the desired result.
\end{proof}

\begin{corollary}\label{cor:var-bd}
Consider $v_k$ defined in~\eqref{eqn:modified-sg}.
Conditioned on~$x_{k-1}$, we have $\E v_k=\nabla F(x_{k-1})$ and 
\[
    \E \|v_k-\nabla F(x_{k-1})\|^2 \leq 4 L_Q \bigl[ P(x_{k-1})-P(x_\star)
    + P(\tilde{x}) - P(x_\star) \bigr].
\]
\end{corollary}

\begin{proof}
Conditioned on~$x_{k-1}$, we take expectation with respect to~$i_k$ to obtain
\[
    \E \left[\frac{1}{n q_{i_k}}\nabla f_{i_k}(x_{k-1})\right]
    = \sum_{i=1}^n \frac{q_i}{n q_i} \nabla f_i(x_{k-1})
    = \sum_{i=1}^n \frac{1}{n} \nabla f_i(x_{k-1})
    =\nabla F(x_{k-1}).
\]
Similarly we have 
$\E \left[(1/(n q_{i_k}))\nabla f_{i_k}(\tilde{x})\right]=\nabla F(\tilde{x})$,
and therefore
\[
\E v_k = \E \left[\frac{1}{n q_{i_k}}\bigl(\nabla f_{i_k}(x_{k-1}) 
    - \nabla f_{i_k}(\tilde{x})\bigr) + \nabla F(\tilde{x}) \right]
    = \nabla F(x_{k-1}).
\]
To bound the variance, we have
\begin{eqnarray*}
\E \|v_k-\nabla F(x_{k-1})\|^2 
&=& \E \biggl\|\frac{1}{n q_{i_k}}\bigl(\nabla f_{i_k} (x_{k-1})-\nabla f_{i_k}(\tilde{x})\bigr) +\nabla F(\tilde{x}) - \nabla F(x_{k-1}) \biggr\|^2 \\
    &=& \E \frac{1}{(n q_{i_k})^2}\|\nabla f_{i_k} (x_{k-1}) - \nabla f_{i_k}(\tilde{x}) \|^2
    - \|\nabla F(x_{k-1}) - \nabla F(\tilde{x})\|^2 \\
    &\leq& \E \frac{1}{(n q_{i_k})^2}\|\nabla f_{i_k} (x_{k-1}) - \nabla f_{i_k}(\tilde{x}) \|^2  \\
    &\leq&  \E \frac{2}{(n q_{i_k})^2}\|\nabla f_{i_k} (x_{k-1}) - \nabla f_{i_k}(x_\star) \|^2 
        + \E \frac{2}{(n q_{i_k})^2}\|\nabla f_{i_k} (\tilde{x}) - \nabla f_{i_k}(x_\star) \|^2 \\
        &=& \frac{2}{n}\sum_{i=1}^n \frac{1}{n q_i}\|\nabla f_i (x_{k-1}) - \nabla f_i(x_\star) \|^2 + \frac{2}{n}\sum_{i=1}^n \frac{1}{n q_i}\|\nabla f_i (\tilde{x}) - \nabla f_i(x_\star) \|^2 \\
    &\leq& 4 L_Q \bigl[ P(x_{k-1})-P(x_\star) + P(\tilde{x}) - P(x_\star) \bigr].
\end{eqnarray*}
In the second equality above, we used the fact that for any random 
vector~$\zeta\in\reals^d$, it holds that
$\E\|\zeta-\E\zeta\|^2 = \E\|\zeta\|^2 - \|\E\zeta\|^2$.
In the second inequality, we used $\|a+b\|^2\leq 2\|a\|^2+2\|b\|^2$. 
In the last inequality, we applied Lemma~\ref{lem:var-bd} twice.
\end{proof}

\subsection{Proof of Theorem~\ref{thm:prox-svrg}}

For convenience, we define the \emph{stochastic gradient mapping} 
\[
    g_k = \frac{1}{\eta}(x_{k-1} - x_k) 
    =\frac{1}{\eta}\left( x_{k-1} - \prox_{\eta R}(x_{k-1} - \eta v_k)\right),
\]
so that the proximal gradient step~\eqref{eqn:prox-svrg-update} can be 
written as
\begin{equation}\label{eqn:psg-grad-map}
    x_k = x_{k-1} - \eta g_k.
\end{equation}

We need the following lemmas in the convergence analysis.
The first one is on the \emph{non-expansiveness} of proximal mapping,
which is well known (see, e.g., \cite[Section~31]{Rockafellar70book}).
\begin{lemma}\label{lem:non-expansive}
    Let $R$ be a closed convex function on $\reals^d$ and $x,y\in\dom(R)$. 
    Then
    \[
        \bigl\|\prox_R(x) - \prox_R(y)\bigr\| \leq \|x-y\|.
    \]
\end{lemma}

The next lemma provides a lower bound of the function $P(x)$ using 
stochastic gradient mapping.
It is a slight generalization of \cite[Lemma~3]{HuKwokPan09nips}, and we give
the proof in Appendix~\ref{apd:composite-lb} for completeness.

\begin{lemma}\label{lem:composite-lb}
    Let $P(x)=F(x)+R(x)$, where $\nabla F(x)$ is Lipschitz continuous with
    parameter~$L$, and $F(x)$ and $R(x)$ has strong convexity parameters 
    $\mu_F$ and $\mu_R$ respectively.
    For any $x\in\dom(R)$ and arbitrary $v\in\reals^d$, define
    \begin{eqnarray*}
        x^+ &=& \prox_{\eta R} (x - \eta v) \\
        g &=& \frac{1}{\eta} (x-x^+) \\
        \Delta &=& v - \nabla F(x),
    \end{eqnarray*}
    where $\eta$ is a step size satisfying $0<\eta\leq 1/L$.
    Then we have for any $y\in\reals^d$,
    \[
        P(y) \geq P(x^+) + g^T(y-x) + \frac{\eta}{2}\|g\|^2 
        + \frac{\mu_F}{2}\|y-x\|^2 + \frac{\mu_R}{2}\|y-x^+\|^2
        + \Delta^T(x^+ - y).
    \]
\end{lemma}

Now we proceed to prove Theorem~\ref{thm:prox-svrg}. 
We start by analyzing how the distance between $x_k$ and $x_\star$ changes
in each iteration.
Using the update rule~\eqref{eqn:psg-grad-map}, we have
\begin{eqnarray*}
\|x_k - x_\star\|^2 
&=& \|x_{k-1} - \eta g_k - x_\star\|^2 \\
&=& \|x_{k-1} - x_\star\|^2 - 2\eta g_k^T(x_{k-1}-x_\star) + \eta^2 \|g_k\|^2 .
\end{eqnarray*}
Applying Lemma~\ref{lem:composite-lb} with $x=x_{k-1}$, $v=v_k$, $x^+=x_k$,
$g=g_k$ and $y=x_\star$, we have
\[
    -g_k^T(x_{k-1}-x_\star) + \frac{\eta}{2}\|g_k\|^2 
    \leq P(x_\star) - P(x_k) - \frac{\mu_F}{2}\|x_{k-1}-x_\star\|^2
    -\frac{\mu_R}{2}\|x_k-x_\star\|^2 - \Delta_k^T(x_k-x_\star),
\]
where $\Delta_k = v_k - \nabla F(x_{k-1})$. 
Note that the assumption in Theorem~\ref{thm:prox-svrg} implies
$\eta<1/(4L_Q)<1/L$ because $L_Q\geq(1/n)\sum_{i=1}^n L_i \geq L$.
Therefore,
\begin{eqnarray}
\|x_k-x_\star\|^2 
&\leq& \|x_{k-1} - x_\star\|^2 
- \eta \mu_F\|x_{k-1}-x_\star\|^2 -\eta\mu_R\|x_k-x_\star\|^2 \nonumber\\
&&- 2\eta [ P(x_k)-P(x_\star)] - 2\eta \Delta_k^T(x_k-x_\star) \nonumber \\
&\leq& \|x_{k-1} - x_\star\|^2 
- 2\eta [ P(x_k)-P(x_\star)] - 2\eta \Delta_k^T(x_k-x_\star) 
\label{eqn:distance-bd}
\end{eqnarray}

Next we upper bound the quantity $-2\eta \Delta_k^T(x_k-x_\star)$.
Although not used in the Prox-SVRG algorithm, 
we can still define the proximal full gradient update as
\[
    \bar{x}_k=\prox_{\eta R} (x_{k-1} - \eta \nabla F(x_{k-1})),
\]
which is independent of the random variable~$i_k$.
Then,
\begin{eqnarray*}
   -2\eta\Delta_k^T(x_k-x_\star) 
&=& -2\eta\Delta_k^T(x_k-\bar{x}_k) -2\eta\Delta_k^T(\bar{x}_k-x_\star) \\
&\leq& 2\eta\|\Delta_k\| \|x_k-\bar{x}_k\|-2\eta\Delta_k^T(\bar{x}_k-x_\star)\\
&\leq& 2\eta\|\Delta_k\| \left\| (x_{k-1}-\eta v_k) - \bigl(x_{k-1} - \eta \nabla F(x_{k-1})\bigr) \right\| -2\eta\Delta_k^T(\bar{x}_k-x_\star) \\
&=& 2\eta^2 \|\Delta_k\|^2 -2\eta\Delta_k^T(\bar{x}_k-x_\star),
\end{eqnarray*}
where in the first inequality we used the Cauchy-Schwarz inequality,
and in the second inequality we used~Lemma~\ref{lem:non-expansive}.
Combining with~\eqref{eqn:distance-bd}, we get
\[
\|x_k-x_\star\|^2 
\leq \|x_{k-1} - x_\star\|^2  - 2\eta [ P(x_k)-P(x_\star)] 
+ 2\eta^2\|\Delta_k\|^2 - 2\eta \Delta_k^T(\bar{x}_k-x_\star).
\]
Now we take expectation on both sides of the above inequality with respect
to~$i_k$ to obtain 
\[
\E \|x_k-x_\star\|^2 
\leq \|x_{k-1} - x_\star\|^2 
 - 2\eta [\,\E P(x_k)-P(x_\star)] 
+ 2\eta^2 \, \E \|\Delta_k\|^2 
- 2\eta \, \E [\Delta_k^T(\bar{x}_k-x_\star)].
\]
We note that both $\bar{x}_k$ and $x_\star$ are independent of the 
random variable $i_k$ and $\E \Delta_k=0$, so 
\[
    \E [\Delta_k^T(\bar{x}_k-x_\star)] = (\E \Delta_k)^T(\bar{x}_k-x_\star)=0.
\]
In addition, we can bound the term $\E \|\Delta_k\|^2$ using 
Corollary~\ref{cor:var-bd} to obtain
\[
\E \|x_k-x_\star\|^2 
\leq \|x_{k-1} - x_\star\|^2  - 2\eta [\,\E P(x_k)-P(x_\star)] 
+ 8 L_Q \eta^2 [ P(x_{k-1})-P(x_\star) + P(\tilde{x}) - P(x_\star) ] .
\]

We consider a fixed stage~$s$, so that $x_0=\tilde{x}=\tilde{x}_{s-1}$
and $\tilde{x}_s=\frac{1}{m}\sum_{k=1}^m x_k$.
By summing the previous inequality over $k=1,\ldots,m$ and taking expectation
with respect to the history of random variables $i_1,\ldots,i_m$, we obtain
\begin{eqnarray*}
&&  \E \|x_m-x_\star\|^2 + 2\eta[\,\E P(x_m)-P(x_\star)] 
    +2\eta(1-4L_Q \eta)\sum_{k=1}^{m-1}[\,\E P(x_k)-P(x_\star)] \\
&\leq& \|x_0-x_\star\|^2 
+ 8 L_Q \eta^2 \bigl[ P(x_0)-P(x_\star) + m(P(\tilde{x})-P(x_\star)) \bigr].
\end{eqnarray*}
Notice that $2\eta(1-4L_Q\eta)<2\eta$ and $x_0=\tilde{x}$,  so we have
\[
    2\eta(1-4L_Q\eta)\sum_{k=1}^m[\,\E P(x_k)-P(x_\star)] 
    \leq \|\tilde{x}-x_\star\|^2 
+ 8 L_Q\eta^2 (m+1) [ P(\tilde{x})-P(x_\star) ].
\]
By convexity of $P$ and definition of $\tilde{x}_s$, we have 
$P(\tilde{x}_s)\leq \frac{1}{m}\sum_{t=1}^m P(x_k)$.
Moreover, strong convexity of~$P$ implies
$\|\tilde{x}-x_\star\|^2 \leq \frac{2}{\mu} [P(\tilde{x})-P(\star)]$.
Therefore, we have
\[
    2\eta(1-4L_Q\eta) m [\, \E P(\tilde{x}_s) - P(x_\star) ]
    \leq \left(\frac{2}{\mu} + 8L_Q\eta^2(m+1)\right)
    [P(\tilde{x}_{s-1}) - P(x_\star)] .
\]
Divide both sides of the above inequality by $2\eta(1-4L_Q\eta) m$, 
we arrive at
\[
\E P(\tilde{x}_s) - P(x_\star) \leq 
\left(\frac{1}{\mu\eta(1-4L_Q\eta)m} + \frac{4L_Q\eta(m+1)}{(1-4L_Q\eta)m}\right)
[P(\tilde{x}_{s-1}) - P(x_\star)].
\]
Finally using the definition of~$\rho$ in~\eqref{eqn:converge-factor}, 
and applying the above inequality recursively, we obtain
\[
    \E P(\tilde{x}_s) - P(x_\star) \leq \rho^s [P(\tilde{x}_0)-P(x_\star)],
\]
which is the desired result.

\section{Numerical experiments}

In this section we present results of several numerical experiments to
illustrate the properties of the Prox-SVRG method,
and compare its performance with several related algorithms. 

We focus on the regularized logistic regression problem for 
binary classification: given a set of training examples 
$(a_1,b_1),\ldots, (a_n, b_n)$ where $a_i\in\reals^d$ and $b_i\in\{+1, -1\}$,
we find the optimal predictor $x\in\reals^d$ by solving
\[
 \minimize_{x\in\reals^d} \quad 
    \frac{1}{n} \sum_{i=1}^n \log\bigl(1+\exp(-b_i a_i^T x)\bigr)
    + \frac{\lambda_2}{2} \|x\|_2^2 + \lambda_1 \|x\|_1 ,
\]
where $\lambda_2$ and $\lambda_1$ are two regularization parameters.
The $\ell_1$ regularization is added to promote sparse solutions.
In terms of the model~\eqref{eqn:min-composite} and~\eqref{eqn:avg-loss}, 
we can have either 
\begin{equation}\label{eqn:L2inF}
    f_i(x)=\log(1+\exp(-b_i a_i^T x)  + (\lambda_2/2)\|x\|_2^2, \qquad
    R(x)=\lambda_1\|x\|_1,
\end{equation}
or
\begin{equation}\label{eqn:L2inR}
    f_i(x)=\log(1+\exp(-b_i a_i^T x), \qquad
    R(x)=(\lambda_2/2)\|x\|_2^2+\lambda_1\|x\|_1,
\end{equation}
depending on the algorithm used.

We used three publicly available data sets.
Their sizes~$n$, dimensions~$d$ as well as sources as listed in
Table~\ref{tab:datasets}.
For \texttt{rcv1} and \texttt{covertype}, we used the processed data
for binary classification from \cite{LIBSVMdata}.
The table also listed the values of~$\lambda_2$ and~$\lambda_1$
that were used in our experiments.
These choices are typical in machine learning benchmarks to obtain 
good classification performance.

\begin{table}[t]
\centering
\begin{tabular}{|c|r|r|c|c|c|}
\hline
data sets & $n\quad$ & $d\quad$ & source & $\lambda_2$ & $\lambda_1$ \\
\hline
\texttt{rcv1}    & 20,242 &  47,236 & \cite{RCV1} & $10^{-4}$ & $10^{-5}$ \\
\texttt{covertype} & 581,012 & 54 & \cite{covertype} & $10^{-5}$ & $10^{-4}$ \\ 
\texttt{sido0}    & 12,678 &  4,932 & \cite{sido} & $10^{-4}$ & $10^{-4}$ \\
\hline
\end{tabular}
\caption{Summary of data sets 
and regularization parameters used in our experiments.}
\label{tab:datasets}
\end{table}

\subsection{Properties of Prox-SVRG}

We first illustrate the numerical characteristics of Prox-SVRG on
the \texttt{rcv1} dataset.
Each example in this dataset has been normalized so that $\|a_i\|_2=1$ for 
all $i=1,\ldots,n$, which leads to the same upper bound on the Lipschitz
constants~$L=L_i=\|a_i\|_2^2/4$.
In our implementation, we used the splitting in~\eqref{eqn:L2inF}
and uniform sampling of the component functions.
We choose the number of stochastic gradient steps~$m$ between full
gradient evaluations as a small multiple of~$n$.

Figure~\ref{fig:rcv1tr-eta} shows the behavior of Prox-SVRG with $m=2n$
when we used three different step sizes.
The horizontal axis is the number of effective passes over the data, where
each effective pass evaluates~$n$ component gradients.
Each full gradient evaluation counts as one effective pass, and appears as
a small flat segment of length~1 on the curves.
It can be seen that the convergence of Prox-SVRG 
becomes slow if the step size is either too big or too small.
The best choice of $\eta=0.1/L$ matches our theoretical analysis
(see the first remark after Theorem~\ref{thm:prox-svrg}).
The number of non-zeros (NNZs) in the iterates $x_k$ converges
quickly to $7237$ after about~$10$ passes over the data.

\begin{figure}[t]
    \centering
    \includegraphics[width=0.49\textwidth]{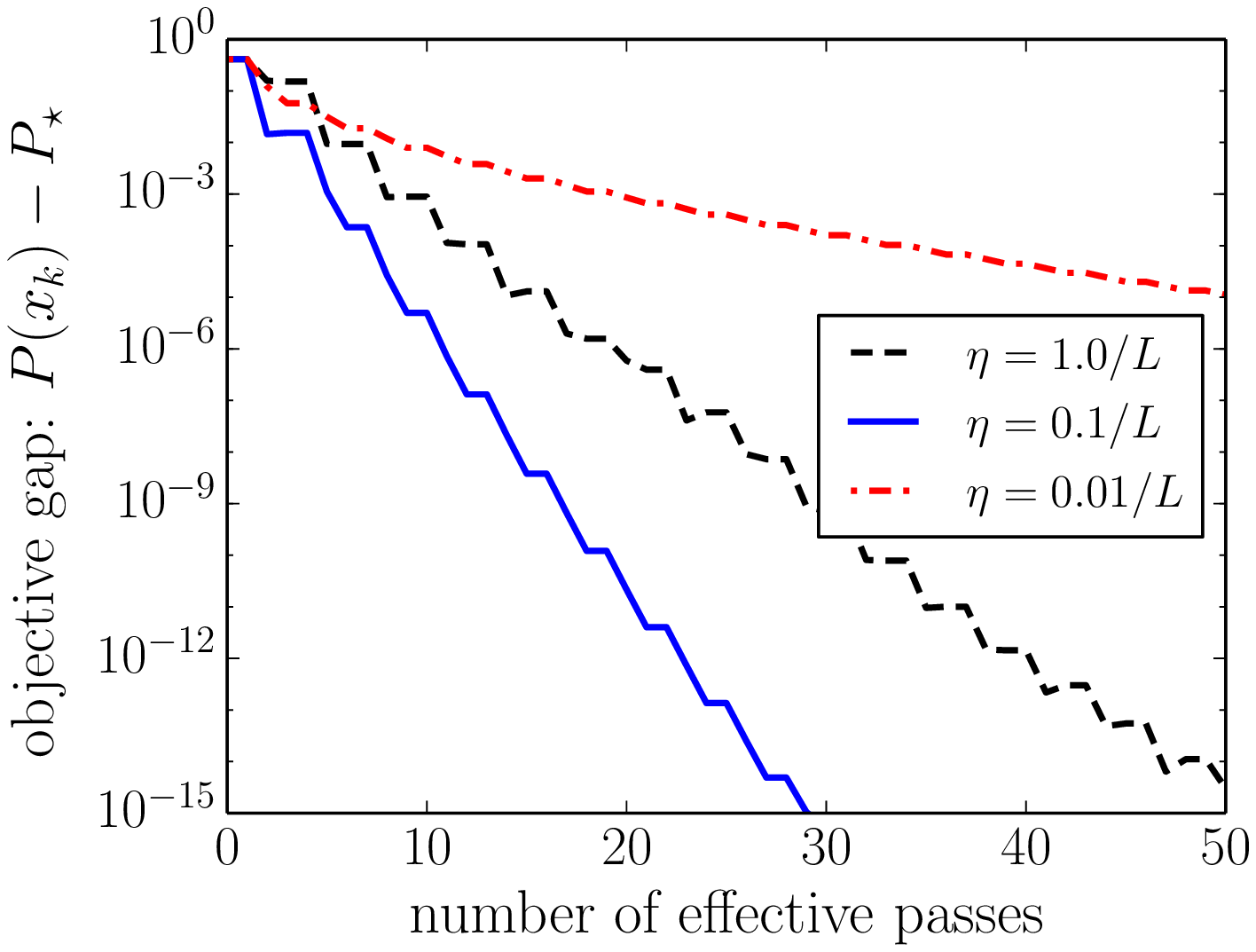}
    \includegraphics[width=0.49\textwidth]{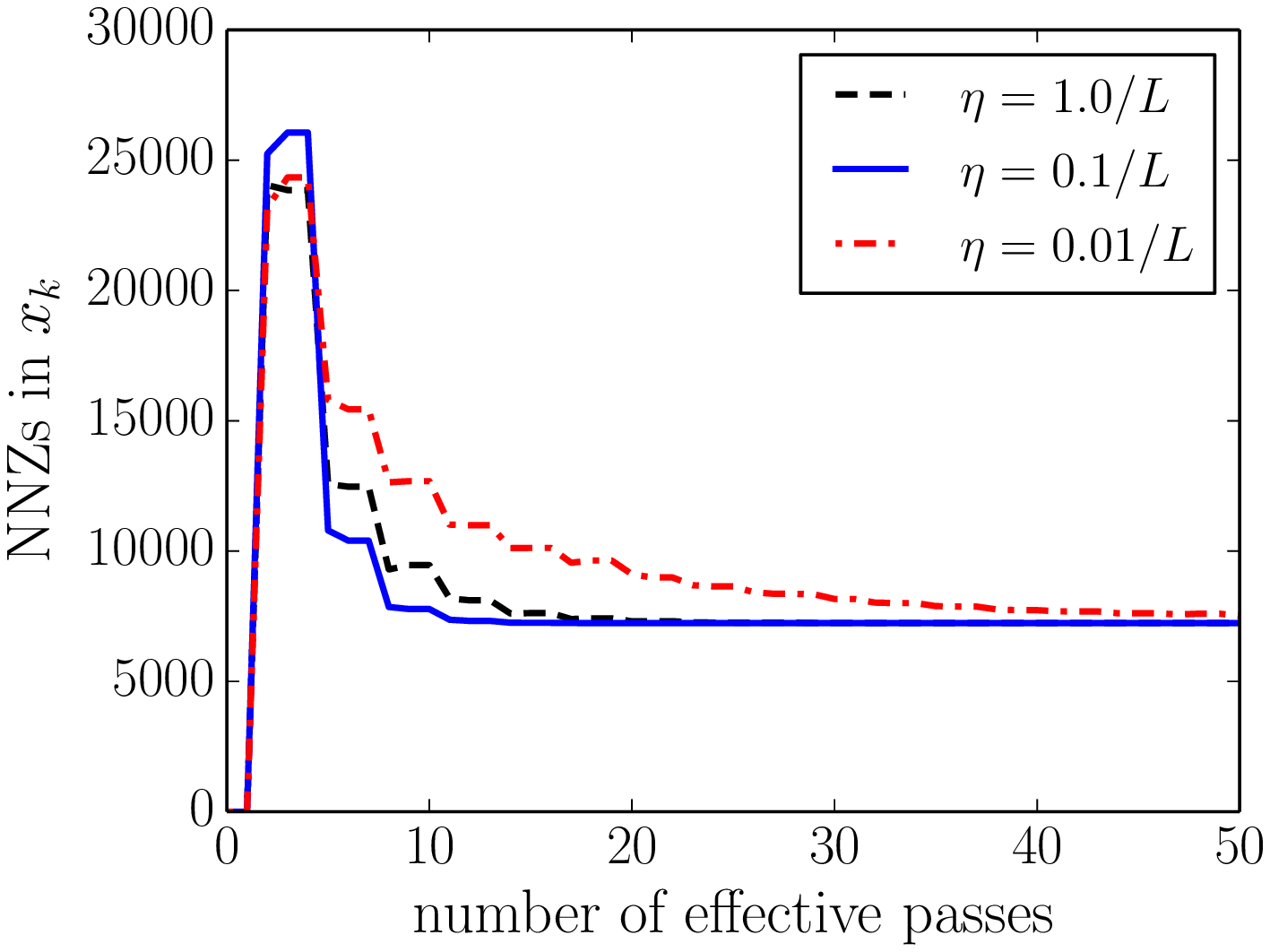}
    \caption{Prox-SVRG on the \texttt{rcv1} dataset: varying the step size~$\eta$ with $m=2n$.}
    \label{fig:rcv1tr-eta}
\end{figure}

\begin{figure}[t]
    \centering
    \includegraphics[width=0.49\textwidth]{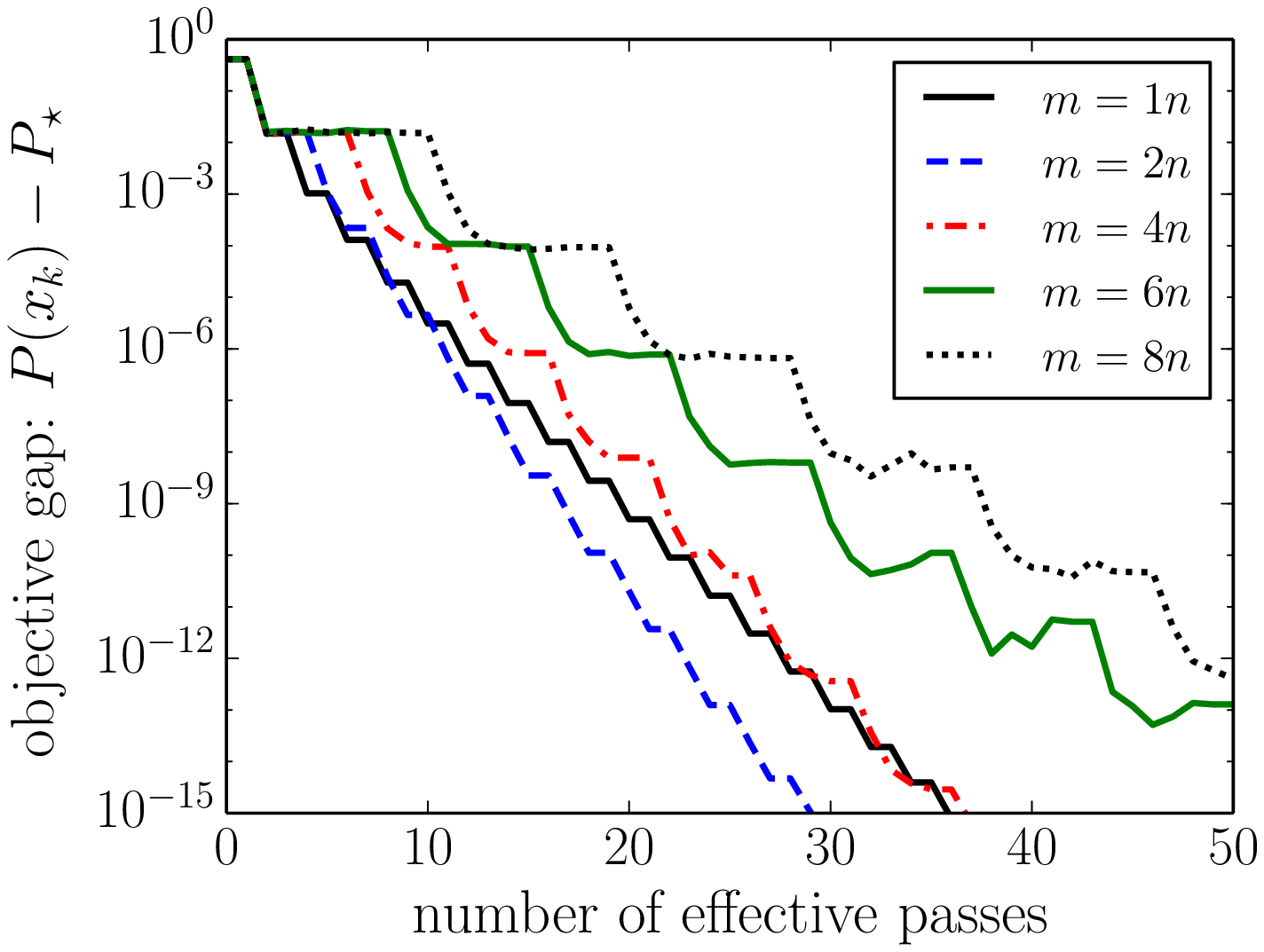}
    \includegraphics[width=0.49\textwidth]{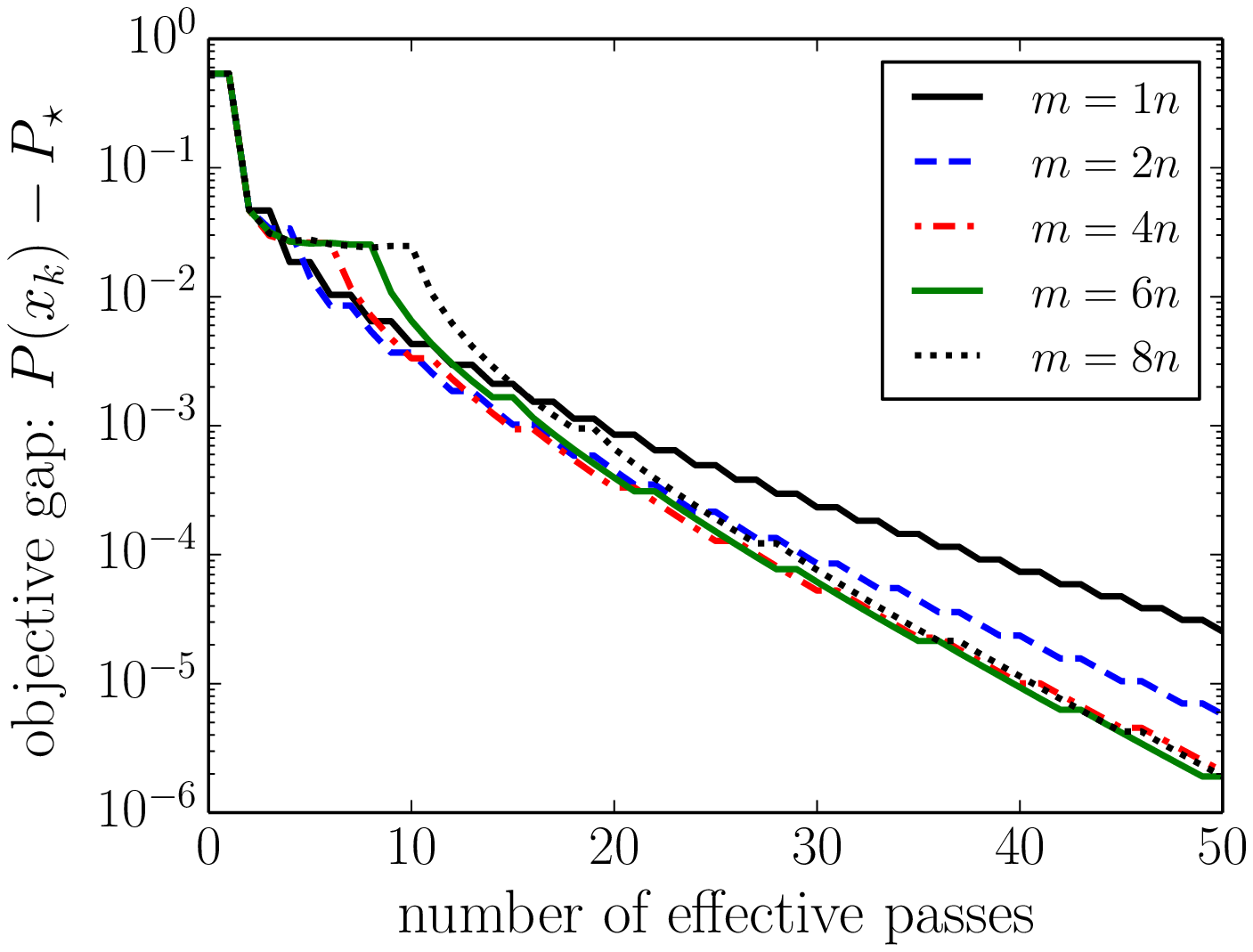}
    \caption{Prox-SVRG on the \texttt{rcv1} dataset with step size
    $\eta=0.1/L$: varying the period~$m$ between full gradient evaluations, 
    with $\lambda_2=10^{-4}$ on the left and $\lambda_2=10^{-5}$ on the right.}
    \label{fig:rcv1tr-vr}
\end{figure}

Figure~\ref{fig:rcv1tr-vr} shows how the objective gap 
$P(x_k)-P_\star$ decreases when we vary the period~$m$ of evaluating
full gradients.
For $\lambda_2=10^{-4}$, the fastest convergence per stage is achieved
by $m=1$, but the frequent evaluation of full gradients makes its overall
performance slightly worse than $m=2$.
Longer periods leads to slower convergence, 
due to the lack of effective variance reduction.
For $\lambda_2=10^{-5}$, the condition number is much larger, thus longer
period~$m$ is required to have sufficient reduction during each stage.

\subsection{Comparison with related algorithms}

\begin{figure}[t]
    \centering
    \includegraphics[width=0.49\textwidth]{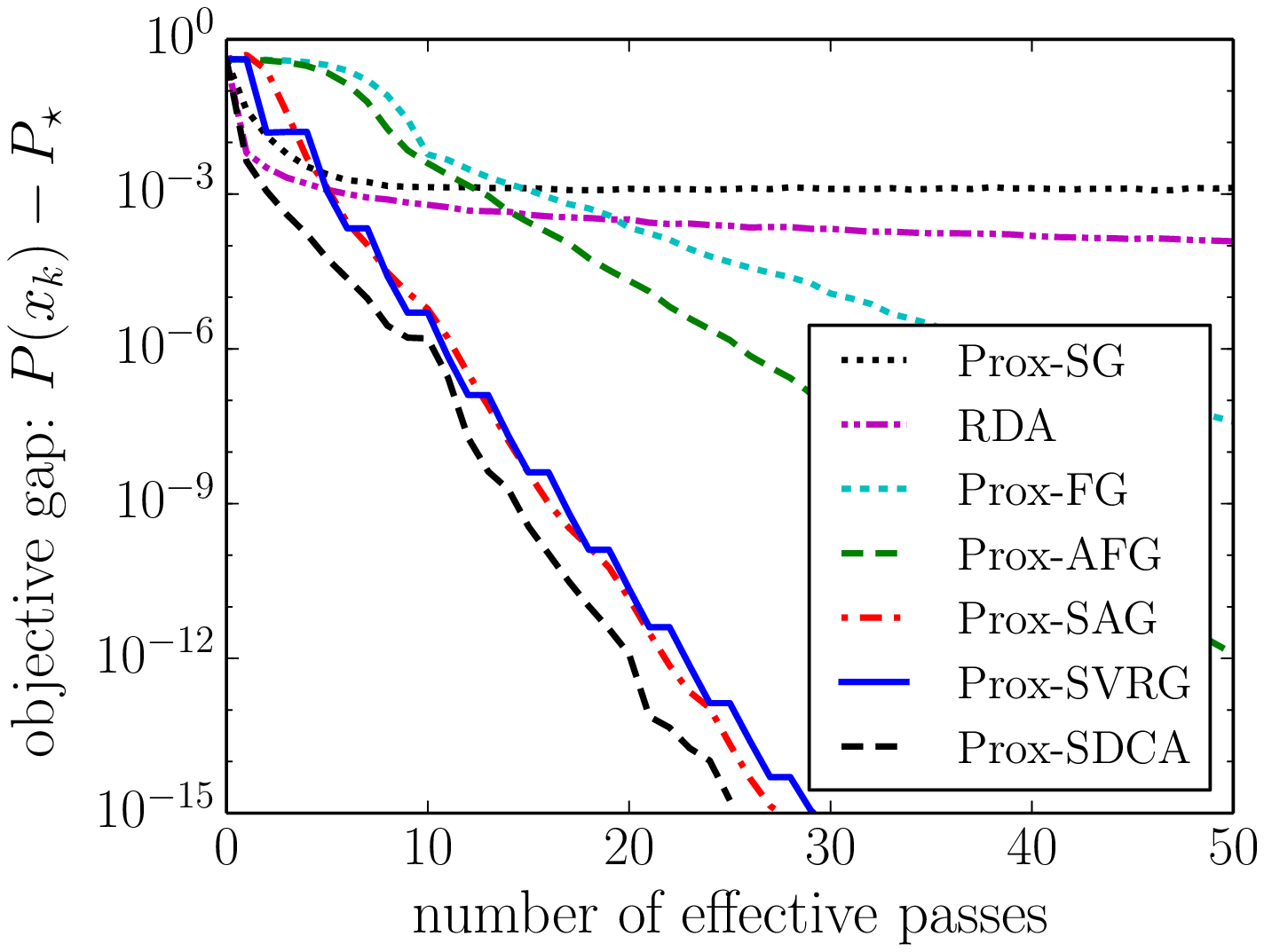}
    \includegraphics[width=0.49\textwidth]{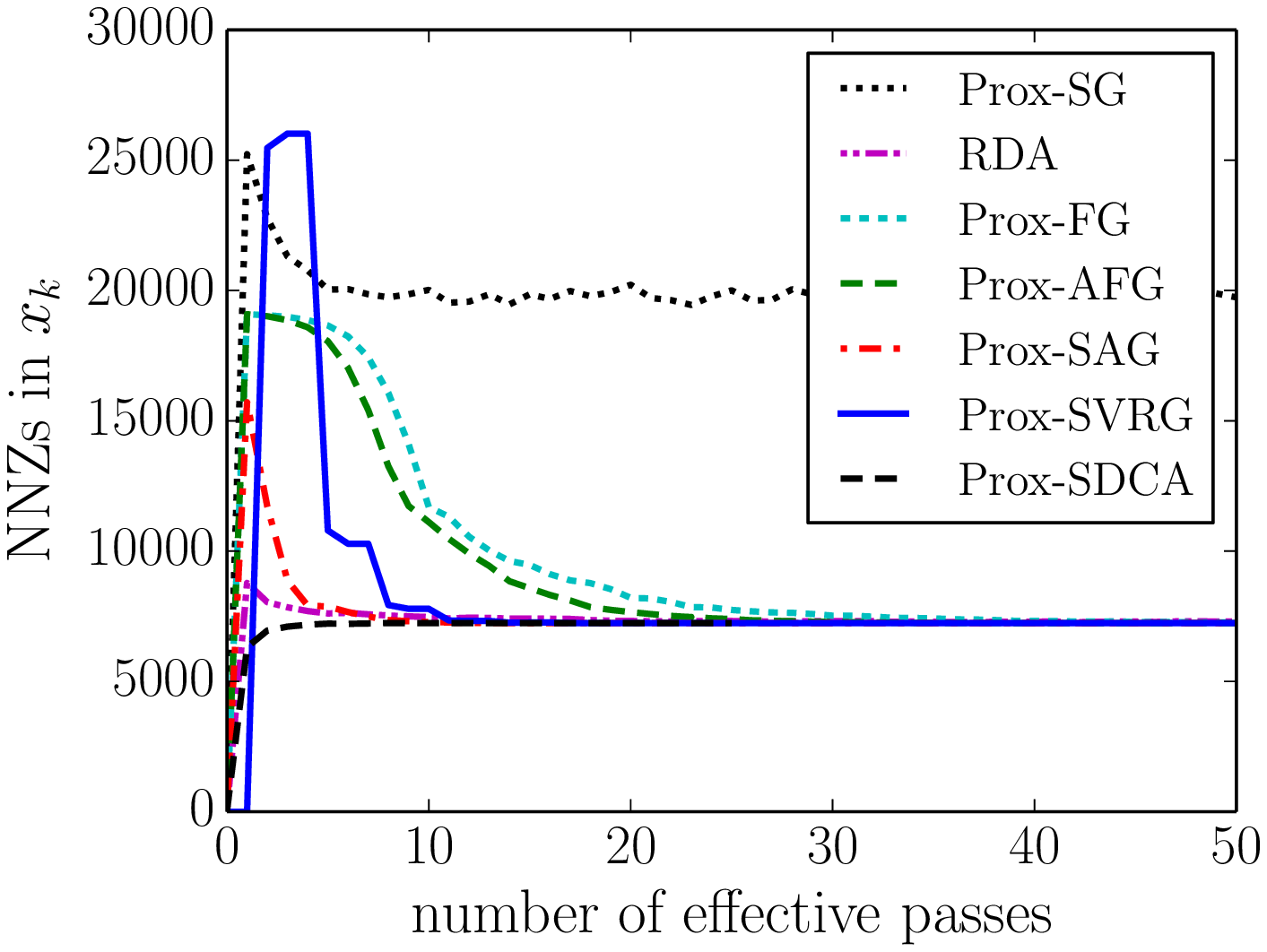}
    \caption{Comparison of different methods on the \texttt{rcv1} dataset.}
    \label{fig:rcv1tr-all}
\end{figure}

We implemented the following algorithms to compare with Prox-SVRG:
\begin{itemize}\itemsep 0pt
    \item Prox-SG: the proximal stochastic gradient method given
        in~\eqref{eqn:prox-sg}.
        We used a constant step size that gave the best performance among 
        all powers of~$10$.
    \item RDA: the regularized dual averaging method in \cite{Xiao2010RDA}.
        The step size parameter~$\gamma$ in RDA is also chosen as the one
        that gave best performance among all powers of~$10$.
    \item Prox-FG: the proximal full gradient method given 
        in~\eqref{eqn:prox-fg}, with an adaptive line search scheme
        proposed in \cite{Nesterov13composite}.
    \item Prox-AFG: an accelerated version of the Prox-FG method that is
        very similar to FISTA \cite{BeckTeboulle09}, also with an 
        adaptive line search scheme.
    \item Prox-SAG: a proximal version of the stochastic average gradient (SAG)
        method \cite[Section~6]{SchmidtLeRouxBach13}.
        We note that the convergence of this Prox-SAG method has not
        been established for the general model considered in this paper.
        Nevertheless it demonstrates good performance in practice.
    \item Prox-SDCA: the proximal stochastic dual coordinate ascent method
        \cite{SSZhang12prox-SDCA}.
        In order to obtain the complexity 
        $O\left( (n+L/\mu)\log(1/\epsilon)\right)$, it needs to use the
        splitting~\eqref{eqn:L2inR}.
\end{itemize}

Figure~\ref{fig:rcv1tr-all} shows the comparison of Prox-SVRG 
($m=2n$ and $\eta=0.1/L$) with different methods described above
on the \texttt{rcv1} dataset.
For the Prox-SAG method, we used the same step size $\eta=0.1/L$ 
as for Prox-SVRG.
We can see that the three methods that performed best are
Prox-SAG, Prox-SVRG and Prox-SDCA. 
The superior performance of Prox-SVRG and Prox-SDCA are predicted by their 
low complexity analysis. 
While the complexity of Prox-SAG has not been formally established,
its performance is among the best.
In terms of obtaining sparse iterates under the $\ell_1$-regularization,
RDA, Prox-SDCA and Prox-SAG converged to the correct NNZs quickly, 
followed by Prox-SVRG and the two full gradient methods.
The Prox-SG method didn't converge to the correct NNZs.

Figure~\ref{fig:cov-sido} shows the comparison of different methods
on two other data sets listed in Table~\ref{tab:datasets}.
Here we also included comparison with Prox-SVRG2,
which is a hybrid method by performing Prox-SG for one pass over the data
and then switch to Prox-SVRG.
This hybrid scheme was suggested in \cite{JohnsonZhang13},
and it often improves the performance of Prox-SVRG substantially.
Similar hybrid schemes also exist for SDCA \cite{SSZhang12prox-SDCA}
and SAG \cite{SchmidtLeRouxBach13}.

The behaviors of the stochastic gradient type of algorithms on 
\texttt{covertype} (Figure~\ref{fig:cov-sido}, left) 
are similar to those on \texttt{rcv1}, but
the two full gradient methods Prox-FG and Prox-AFG perform worse
because of the smaller regularization parameter~$\lambda_2$ and hence
worse condition number.
The \texttt{sido0} data set turns out to be more difficult to optimize,
and much slower convergence are observed in Figure~\ref{fig:cov-sido} (right).
The Prox-SAG method performs best on this data set, followed by
Prox-SVRG2 and Prox-SVRG.

\begin{figure}[t]
    \centering
    \includegraphics[width=0.49\textwidth]{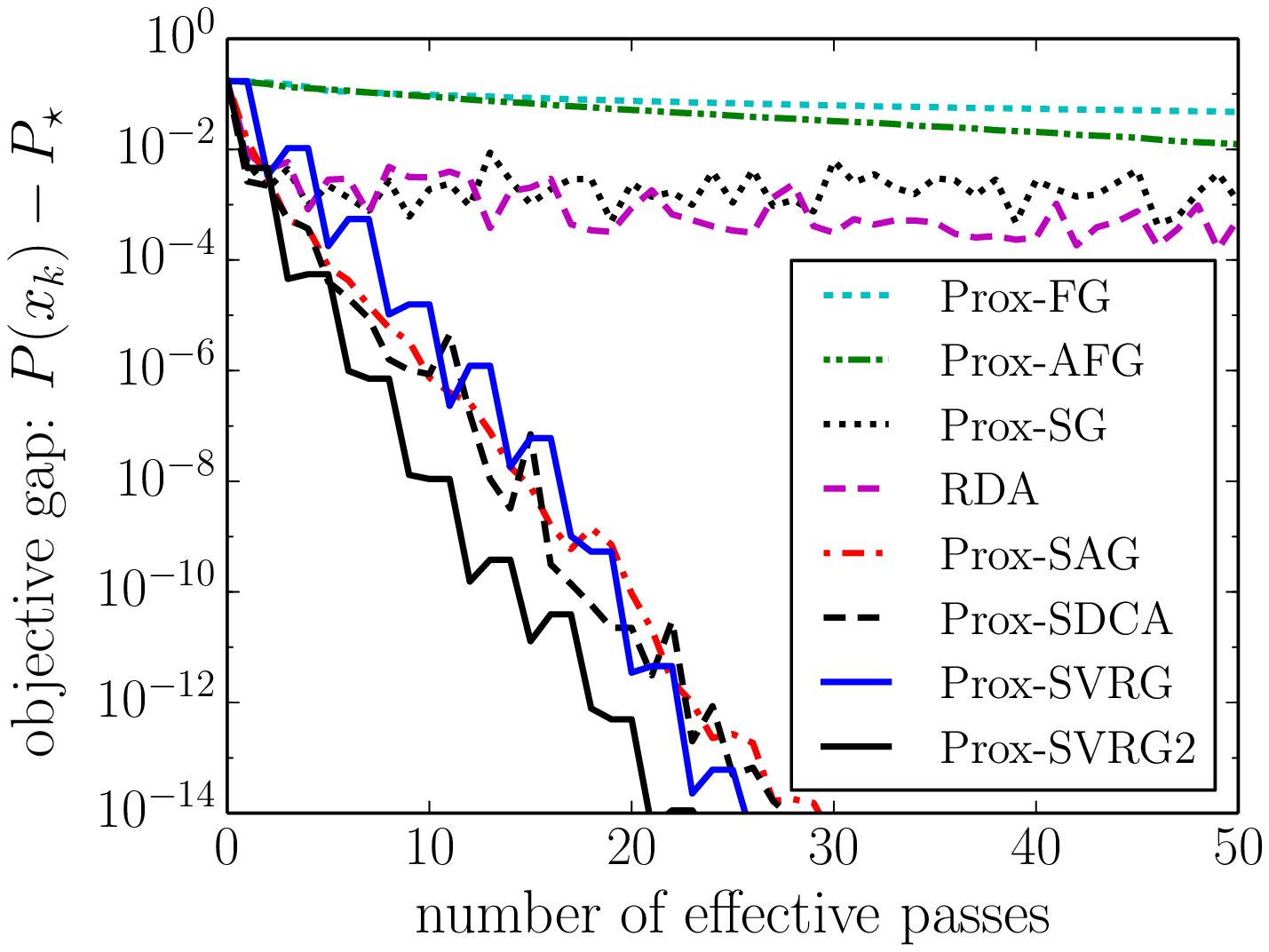}
    \includegraphics[width=0.49\textwidth]{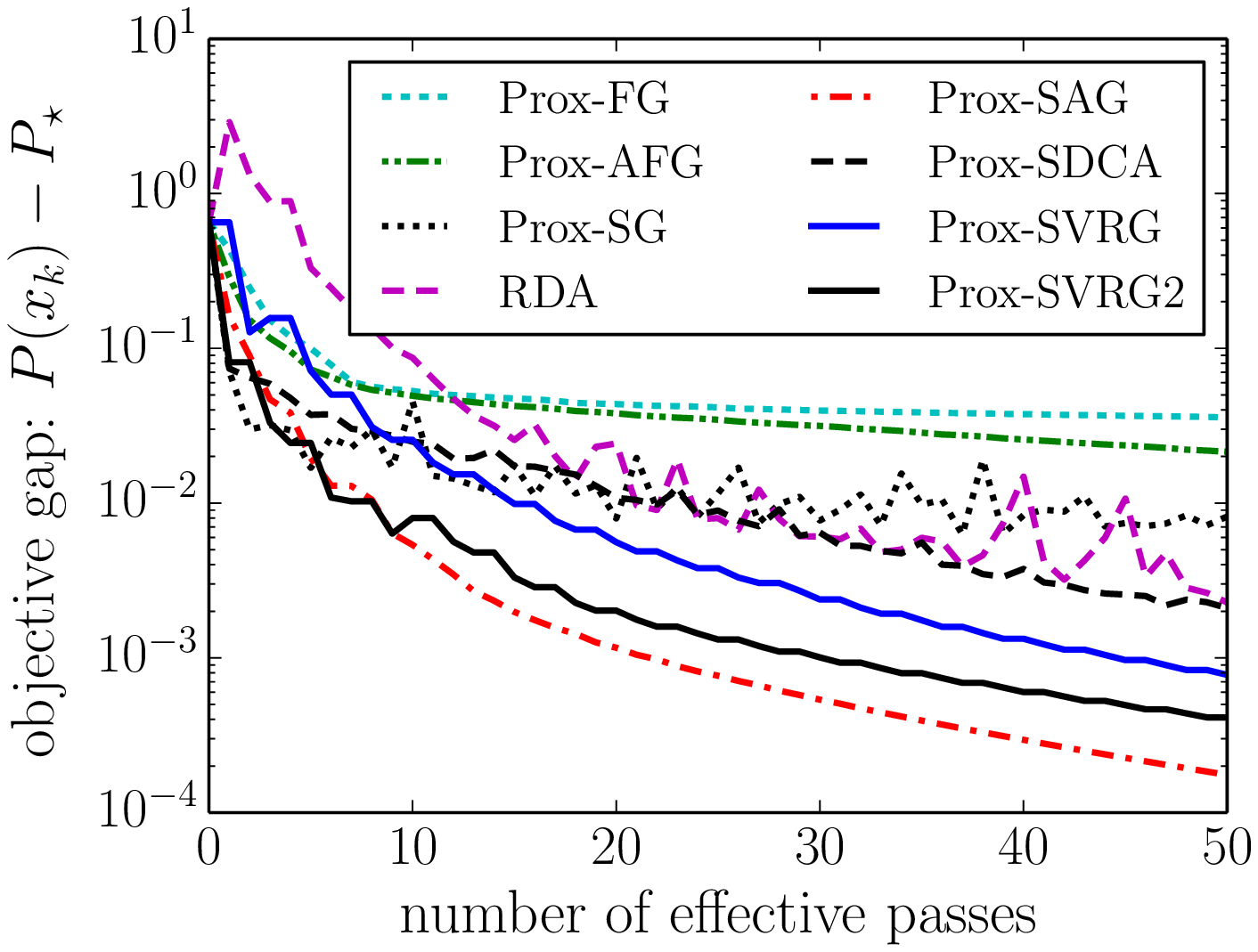}
    \caption{Comparison of different methods on \texttt{covertype} (left)
    and \texttt{sido0} (right).}
    \label{fig:cov-sido}
\end{figure}

\section{Conclusions}
We developed a new proximal stochastic gradient method, called Prox-SVRG,
for minimizing the sum of two convex functions: one is the average of a large 
number of smooth component functions, and the other is a general convex 
function that admits a simple proximal mapping. 
This method exploits the finite average structure of the smooth part by
extending the variance reduction technique of SVRG \cite{JohnsonZhang13}, 
which computes the full gradient periodically to modify the stochastic 
gradients in order to reduce their variance.

The Prox-SVRG method enjoys the same low complexity as that of  
SDCA \cite{SSZhang13SDCA, SSZhang12prox-SDCA} and 
SAG \cite{LeRouxSchmidtBach12,SchmidtLeRouxBach13},
but applies to a more general class of problems, 
and does not require the storage of
the most recent gradient for each component function.
In addition, our method incorporates a weighted sampling scheme, 
which achieves an improved complexity result for problems where 
the component functions vary substantially in smoothness.


\appendix

\section{Proof of Lemma~\ref{lem:composite-lb}}
\label{apd:composite-lb}

We can write the proximal update $x^+=\prox_{\eta R}(x-\eta v)$ 
more explicitly as
\[
    x^+ = \argmin_y \left\{\frac{1}{2}\|y - (x-\eta v)\|^2
          + \eta R(y) \right\}.
\]
The associated optimality condition states that
there is a $\xi\in\partial R(x^+)$ such that
\[
   x^+ - ( x - \eta v ) + \eta \xi = 0.
\]
Combining with the definition of $g=(x-x^+)/\eta$, we have $\xi = g - v$.

By strong convexity of $F$ and $R$, we have for any $x\in\dom(R)$ and
$y\in\reals^d$,
\begin{eqnarray*}
    P(y) &=& F(y) + R(y) \\
    &\geq& F(x) + \nabla F(x)^T (y-x) + \frac{\mu_F}{2} \|y-x\|^2
        + R(x^+) + \xi^T(y - x^+) + \frac{\mu_R}{2} \|y-x^+\|^2 .
\end{eqnarray*}
By smoothness of~$F$, we can further lower bound $F(x)$ by
\[
F(x)\geq F(x^+) - \nabla F(x)^T(x^+-x) - \frac{L}{2}\|x^+ -x\|^2 .
\]
Therefore,
\begin{eqnarray*}
P(y) &\geq& F(x^+) - \nabla F(x)^T(x^+-x) - \frac{L}{2}\|x^+-x\|^2 \\
&& + \nabla F(x)^T (y-x) + \frac{\mu_F}{2} \|y-x\|^2
       + R(x^+) + \xi^T(y - x^+) + \frac{\mu_R}{2} \|y-x^+\|^2 \\
&=& P(x^+) - \nabla F(x)^T(x^+-x) - \frac{L\eta^2}{2}\|g\|^2 \\
&& + \nabla F(x)^T (y-x) + \frac{\mu_F}{2} \|y-x\|^2
       + \xi^T(y - x^+) + \frac{\mu_R}{2} \|y-x^+\|^2  ,
\end{eqnarray*}
where in the last equality we used $P(x^+)=F(x^+)+R(x^+)$ and $x^+-x=-\eta g$.
Collecting all inner products on the right-hand side, we have
\begin{eqnarray*}
&&  -\nabla F(x)^T(x^+-x) + \nabla F(x)^T(y-x) + \xi^T(y-x^+) \\
&=& \nabla F(x)^T(y-x^+) + (g-v)^T(y-x^+)\\
&=& g^T(y-x^+) + (v-\nabla F(x))^T(x^+-y) \\
&=& g^T(y-x + x-x^+) + \Delta^T(x^+-y) \\
&=& g^T(y-x) + \eta\|g\|^2 + \Delta^T(x^+-y) ,
\end{eqnarray*}
where in the first equality we used $\xi=g-v$, 
in the third equality we used $\Delta=v-\nabla F(x)$, 
and in the last equality we used $x-x^+=\eta g$.
Putting everything together, we obtain
\[
  P(y) \geq P(x^+) + g^T(y-x) + \frac{\eta}{2}(2-L\eta)\|g\|^2 
  + \frac{\mu_F}{2}\|y-x\|^2 + \frac{\mu_R}{2}\|y-x^+\|^2
  + \Delta^T(x^+ - y).
\]
Finally using the assumption $0<\eta\leq 1/L$, we arrive at the desired result.

\section{Convergence analysis of the Prox-FG method}
\label{apd:prox-fg}

Here we prove the convergence rate in~\eqref{eqn:prox-fg-rate}
for the Prox-FG method~\eqref{eqn:prox-fg}.
First we define the full gradient mapping 
$G_k = (x_k-x_{k-1})/\eta$ and use it to obtain
\begin{eqnarray*}
    \|x_k-x_\star\|^2 &=& \|x_{k-1}-x_\star-\eta G_k\|^2 \\
    &=& \|x_{k-1}-x_\star\|^2 - 2\eta G_k^T (x_{k-1}-x_\star) 
      + \eta^2 \|G_k\|^2 . 
\end{eqnarray*}
Applying Lemma~\ref{lem:composite-lb} with $x=x_{k-1}$, $v=\nabla F(x_{k-1})$,
$x^+=x_k$, $g=G_k$ and $y=x_\star$, we have $\Delta=0$ and
\[
    -G_k^T(x_{k-1}-x_\star) + \frac{\eta}{2}\|G_k\|^2 
    \leq P(x_\star) - P(x_k) - \frac{\mu_F}{2}\|x_{k-1}-x_\star\|^2
    -\frac{\mu_R}{2}\|x_k-x_\star\|^2 . 
\]
Therefore,
\begin{eqnarray*}
\|x_k-x_\star\|^2 
\leq \|x_{k-1}-x_\star\|^2 + 2\eta \left(F(x_\star)-F(x_k) 
-\frac{\mu_F}{2} \|x_{k-1}-x_\star\|^2 
 - \frac{\mu_R}{2} \|x_k-x_\star\|^2\right) .
\end{eqnarray*}
Rearranging terms in the above inequality yields
\begin{equation}\label{eqn:fg-obj-dist}
2\eta \bigl(F(x_k)-F(x_\star)\bigr) + (1+\eta\mu_R) \|x_k-x_\star\|^2 
\leq (1-\eta\mu_F) \|x_{k-1}-x_\star\|^2 .
\end{equation}
Dropping the nonnegative term $2\eta \bigl(F(x_k)-F(x_\star)\bigr)$ on the
left-hand side results in
\[
\|x_k-x_\star\|^2 \leq \frac{1-\eta\mu_F}{1+\eta\mu_R} \|x_{k-1}-x_\star\|^2,
\]
which leads to
\[
\|x_k-x_\star\|^2 \leq \left(\frac{1-\eta\mu_F}{1+\eta\mu_R}\right)^k 
\|x_0-x_\star\|^2. 
\]
Dropping the nonnegative term $(1+\eta\mu_R) \|x_k-x_\star\|^2$ on the left-hand
side of~\eqref{eqn:fg-obj-dist} yields
\[
F(x_k)-F(x_\star) \leq \frac{1-\eta\mu_F}{2\eta} \|x_{k-1}-x_\star\|^2 
\leq \frac{1+\eta\mu_R}{2\eta} \left(\frac{1-\eta\mu_F}{1+\eta\mu_R}\right)^k 
\|x_0-x_\star\|^2. 
\]
Setting $\eta=1/L$, the above inequality is equivalent to~\eqref{eqn:prox-fg-rate}.

\bibliographystyle{alpha}
\bibliography{prox_svrg}

\end{document}